\documentclass[11pt]{article}
\usepackage{amssymb,amsmath,mathrsfs,amsthm,indentfirst} % font used for R in Real numbers
\usepackage[numbers,sort&compress]{natbib}
\DeclareMathOperator\dif{d\!} %微分

\begin{document}

\title{
\bf Analysis of a nonlinear free-boundary tumor model with angiogenesis and a connection between the nonnecrotic and necrotic phases}

\author{Huijuan Song, Wentao Hu, Zejia Wang\thanks{Corresponding author.
Email: zejiawang@jxnu.edu.cn}
\\
{\small \it School of Mathematics and Statistics, Jiangxi Normal University, }
\\
{\small \it Nanchang, Jiangxi 330022, PR China}
}

\date{}

\maketitle

\begin{abstract}
This paper is concerned with a nonlinear free boundary problem modeling the growth of spherically symmetric tumors with angiogenesis, set with a Robin boundary condition. In which, both nonnecrotic tumors and necrotic tumors are taken into consideration. The well-posedness and asymptotic behavior of solutions are studied. It is shown that there exist two thresholds, denoted by $\tilde\sigma$ and $\sigma^*$, on the surrounding nutrient concentration $\bar\sigma$. If $\bar\sigma\leq\tilde\sigma$, then the considered problem admits no stationary solution and all evolutionary tumors will finally vanish, while if $\bar\sigma>\tilde\sigma$, then it admits a unique stationary solution and all evolutionary tumors will converge to this dormant tumor; moreover, the dormant tumor is nonnecrotic if $\tilde\sigma<\bar\sigma\leq\sigma^*$ and necrotic if $\bar\sigma>\sigma^*$.
The connection and mutual transition between the nonnecrotic and necrotic phases are also given.
\\
{2010MSC:} 35R35, 35B35, 35Q92
\end{abstract}

\let\oldsection\section
\def\SEC{\oldsection}
\renewcommand\section{\setcounter{equation}{0}\SEC}
\renewcommand\thesection{\arabic{section}}

\renewcommand\theequation{\thesection.\arabic{equation}}
\newtheorem{proposition}{Proposition}[section]
\newtheorem{lemma}{Lemma}[section]
\newtheorem{theorem}{Theorem}[section]
\newtheorem{remark}{Remark}[section]
\newtheorem{corollary}{Corollary}[section]
\def\pd#1#2{\dfrac{\partial#1}{\partial#2}}
\allowdisplaybreaks
\renewcommand{\proofname}{\indent\it\bfseries Proof}

\section{Introduction}

In this paper, we study the following nonlinear free boundary problem modeling the growth of spherically symmetric tumors with angiogenesis:
\begin{align}
    &\frac1{r^2}\frac{\partial}{\partial r}\left(r^2\frac{\partial\sigma}{\partial r}\right)= f(\sigma)H(\sigma-\sigma_D)\quad\text{for}~0<r<R(t),~t>0,
    \label{eq(1.1)}
    \\
    &\frac{\partial\sigma}{\partial r}(0,t)=0\quad\text{for}~t>0,
    \label{eq(1.2)}
    \\
    &\frac{\partial\sigma}{\partial r}(R(t),t)+\beta[\sigma(R(t),t)-\bar\sigma]=0\quad\text{for}~t>0,
    \label{eq(1.3)}
    \\
    &R^2(t)\frac{\dif R}{\dif t}=\int_{\sigma>\sigma_D}g(\sigma)r^2\dif r-\int_{\sigma\leq\sigma_D}\nu r^2\dif r\quad\text{for}~t>0,
    \label{eq(1.4)}
    \\
    &R(0)=R_0,
    \label{eq(1.5)}
\end{align}
where $\sigma(r,t)$ is the nutrient concentration in the tumor, $r=|x|$, $x\in\mathbb{R}^3$, and $R(t)$ is the radius of the tumor at time $t$, which are both unknown and need to be determined together, $H(s)$ is the Heaviside function: $H(s)=0$ if $s\leq0$ and $H(s)=1$ if $s>0$, $\sigma_D$ is a positive constant representing the threshold value such that in the region
where $\sigma>\sigma_D$ nutrient is enough to sustain (at least a portion of) tumor cells alive
and proliferating, whereas in the region where $\sigma\leq\sigma_D$ (which may be empty) nutrient is insufficient to
sustain any tumor cell alive, $f(\sigma)$ and $g(\sigma)$ are given nutrient consumption rate and tumor cell proliferation rate functions, respectively, $\bar\sigma$, $\beta$ and $\nu$ are positive constants representing the nutrient concentration outside the tumor, the rate of nutrient supply to the tumor, and the dissolution rate of necrotic cells, respectively, and $R_0>0$ is a given initial tumor radius.

Angiogenesis is a complex process in which tumor cells secrete substances that promote the formation of new blood vessels penetrating into the tumor. In this model, the nutrient enters the tumor through these new blood vessels and the impact from angiogenesis is incorporated in the boundary condition \eqref{eq(1.3)}, where the positive constant $\beta$
reflects the strength of the blood vessel system of the tumor; the smaller $\beta$ is, the weaker the blood vessel system of the tumor will be;
$\beta=0$ means that the tumor does not have its own blood vessel system and $\beta=\infty$ indicates that the tumor
is all surrounded by the blood vessels which reduces to the Dirichlet boundary condition
\begin{equation}
\label{eq(1.7)}
\sigma(R(t),t)=\bar\sigma\quad\text{for}~t>0.
\end{equation}
Thus, from biological viewpoint, the boundary condition \eqref{eq(1.3)} is more realistic compared with \eqref{eq(1.7)}. For more discussions please see \cite{F-L(15),Z-C(18), H-Z-H(20)}.

Before going to our interest, we prefer to recall some relevant works. The above model with linear consumption and proliferation rates:
\begin{equation}
\label{eq(1.6)}
f(\sigma)=\sigma,\quad g(\sigma)=\mu(\sigma-\tilde\sigma)\quad(\mu>0,~\tilde\sigma>0)
\end{equation}
and the Dirichlet boundary condition \eqref{eq(1.7)}, was proposed by Cui \cite{Cui(06)} as in essence a combination of two Byrne-Chaplain inhibitor-free and avascular tumor models; see \cite{B-C(95)} for the nonnecrotic case and then \cite{B-C(96)} for its necrotic version.
This is made such that both nonnecrotic tumors and necrotic tumors can be considered in a joint way. By delicate calculations based on the existence of an explicit form for solutions of \eqref{eq(1.1)},
\eqref{eq(1.2)}, \eqref{eq(1.7)}, \eqref{eq(1.6)} for given $R(t)$, Cui \cite{Cui(06)}
studied the existence, uniqueness and global asymptotic stability of stationary solutions, the dependence on the parameters $\nu$ and $\bar\sigma$, as well as the mutual transition between the nonnecrotic and necrotic phases. Recently, Wu and Wang \cite{W-W(19)} extended the results to the case of general nonlinear functions $f$ and $g$. Since no explicit solution is available, they investigated much more profound relations between all unknown functions. For a broader discussion on the proliferation rate in the nonnecrotic case, we refer the reader to \cite{B-E-Z(08)}. More recently, Wu and Xu \cite{W-X(20)} further established similar results for the nonlinear model with a periodic external nutrient supply, that is
$$
\sigma(R(t),t)=\phi(t)\quad\text{for}~t>0,
$$
where $\phi(t)$ is a positive periodic function. For the analysis of other related tumor models, see \cite{B-F(05),F-MB(05),Z-X(14)} for example.
When replacing the Dirichlet boundary condition \eqref{eq(1.7)} with the Robin condition \eqref{eq(1.3)}, the model \eqref{eq(1.1)}-\eqref{eq(1.5)} with linear functions \eqref{eq(1.6)} was recently studied by Xu and Su \cite{X-S(20)}. As for nonlinear functions $f$ and $g$, the nonnecrotic case
was analyzed by Zhuang and Cui \cite{Z-C(18)}.

Motivated by the above works, in this paper we aim at studying the problem \eqref{eq(1.1)}-\eqref{eq(1.5)} with nonlinear $f$, $g$
and the Robin boundary condition, and revealing the dependence of solutions on the parameter $\bar\sigma$. It should be pointed out that $\bar\sigma$ rather than other model parameters is chosen as the dependent parameter, because we think it is extrinsic for tumors. As in \cite{W-W(19)} we assume

(A1) $f \in C^1[0,+\infty)$, $f'$ is positive and bounded on $[0,+\infty)$, and $f(0)=0$;

(A2) $g \in C^1[0,+\infty)$, $g'\geq0$ on $[0,+\infty)$ and does not identically equal zero in any interval, and $g(\tilde\sigma)=0$ for some $\tilde\sigma>0$;

(A3) $\sigma_D<\min\{\tilde\sigma,\bar\sigma\}$ and $g(\sigma_D)+\nu \geq 0$.\\
Here the relation $g(\sigma_D)+\nu \geq 0$ means that the volume loss rate of living cells at $\sigma_D$ is not greater than the dissolution rate of necrotic cells (see \cite{Cui(06)} for the detailed derivation).

Based on the above assumptions, it is shown that there exist two thresholds on $\bar\sigma$: $\tilde\sigma$ and $\sigma^*$ (see \eqref{eq(2.41)}). Indeed, if $\bar\sigma\leq\tilde\sigma$, then the problem \eqref{eq(1.1)}-\eqref{eq(1.5)} admits no stationary solution and all evolutionary tumors will finally vanish; if $\bar\sigma>\tilde\sigma$, then \eqref{eq(1.1)}-\eqref{eq(1.5)} admits a unique stationary solution and all evolutionary tumors will converge to this dormant tumor; moreover, the dormant tumor is nonnecrotic if $\tilde\sigma<\bar\sigma\leq\sigma^*$ and necrotic if $\bar\sigma>\sigma^*$ (see Theorems \ref{thm-1} and \ref{thm-2}). Besides, it is found that mutual transition may exist between the nonnecrotic and necrotic phases in the growth of tumors (see Theorem \ref{thm-3}). The biological implication of these results is that the existence, structure (nonnecrotic or necrotic) and stability of dormant tumor state
can be controlled by external nutrient concentration.

For the nonlinear problem, the analysis of stationary solutions was usually made by first transforming the free boundary problem into an equivalent problem with fixed boundary and then investigating this fixed boundary problem. However, some new difficulties being different from those encountered in the Dirichlet boundary condition problem arise when we tackle the Robin boundary condition problem. For example, the maximum principle may not work sometimes. Inspired by \cite{Z-C(18)}, we solve it by employing the relationship between the transformed function and its original function, and applying the maximum principle to an auxiliary problem; see the proof of \eqref{eq(2.00)}.

The rest of the paper is organized as follows. In Section 2, we establish the existence and uniqueness of stationary solutions to the problem \eqref{eq(1.1)}-\eqref{eq(1.5)}. In Section 3, we prove the global existence and asymptotic behavior of transient solutions, and present the connection and mutual transition between the nonnecrotic and necrotic phases.

%%%%%%%%%%%%%%%%%%%%%%%%%%%%%%%%%%%%%%%%%%%%%%%%%%%%%%%%%%%%%%%%%%%%%%%%%%%%%%%%%%%%%%%%
\section{Stationary solutions}

In this section, we study stationary solutions to the system \eqref{eq(1.1)}-\eqref{eq(1.5)}, denoted by $(\sigma_s(r),R_s)$, which satisfy
\begin{align}
    &\sigma''(r)+\frac{2}{r}\sigma'(r)= f(\sigma)H(\sigma-\sigma_D)\quad\text{for}~0<r<R,
    \label{eq-1}
    \\
    &\sigma'(0)=0,
    \label{eq-2}
    \\
    &\sigma'(R)+\beta(\sigma(R)-\bar\sigma)=0,
    \label{eq-3}
    \\
    &\frac1{R^3}\left[\int_{\sigma>\sigma_D}g(\sigma(r))r^2\dif r-\int_{\sigma\leq\sigma_D}\nu r^2\dif r\right]=0.
    \label{eq-4}
\end{align}
To be more specific, if the dormant tumor has a necrotic core with radius $\rho\in(0,R)$, then the problem \eqref{eq-1}-\eqref{eq-4} becomes
\begin{align}
    &\sigma''(r)+\frac{2}{r}\sigma'(r)=f(\sigma) \quad\text{for}~\rho<r<R,
    \label{eq(2.1)}
    \\
    &\sigma'(\rho)=0,\quad\sigma'(R)+\beta(\sigma(R)-\bar\sigma)=0,
     \label{eq(2.2)}
    \\
    &\sigma(r)=\sigma_D \quad \text{for}~0\leq r\leq\rho,
    \label{eq(2.3)}
    \\
    &\frac1{R^3}\left[\int_\rho^R g(\sigma(r))r^2\dif r-\frac{\nu}{3}\rho^3\right]=0,
     \label{eq(2.4)}
\end{align}
whereas if the dormant tumor does not have a necrotic core, then \eqref{eq-1}-\eqref{eq-4} reduces to
\begin{align}
    &\sigma''(r)+\frac{2}{r}\sigma'(r)=f(\sigma) \quad\text{for}~0<r<R,
    \label{eq(2.03)}
    \\
    &\sigma'(0)=0,\quad \sigma(0)\geq\sigma_D,
      \label{eq(2.04)}
    \\
    &\sigma'(R)+\beta(\sigma(R)-\bar\sigma)=0,
      \label{eq(2.05)}
    \\
    &\frac1{R^3}\int_0^R g(\sigma(r))r^2\dif r=0.
      \label{eq(2.06)}
\end{align}
For any given $R>0$, setting $s=\frac{r}{R}$, $\eta=\frac{\rho}{R}$ and $u(s)=\sigma(r)$, \eqref{eq(2.1)} and \eqref{eq(2.2)} are transformed into
\begin{align}
    &u''(s)+\frac{2}{s}u'(s)=R^2f(u) \quad \text{for}~\eta<s<1,
    \label{eq(2.5)}
    \\
    &u'(\eta)=0,\quad u'(1)+\beta R(u(1)-\bar\sigma)=0.
    \label{eq(2.6)}
\end{align}

\begin{lemma}
\label{lem-2.1}
Let the assumptions (A1) and (A3) hold. Then for any $R>0$ and any $0\leq\eta<1$, the problem \eqref{eq(2.5)}, \eqref{eq(2.6)} allows a unique solution $u(s)=U(s,\eta,R)$. Moreover,

(i) for $\eta\leq s\leq 1$,
\begin{equation}
\label{eq(2.0)}
0<U(s,\eta,R)<\bar\sigma;
\end{equation}

(ii) for $\eta<s\le1$,
\begin{equation}
    \label{eq(2.7)}
    0<\frac{\partial U}{\partial s}(s,\eta,R)\leq\frac{s R^2 f(U(s,\eta,R))}{3}
\end{equation}
and
\begin{equation}
    \label{eq(2.8)}
    \frac{1}{s}\frac{\partial U}{\partial s}(s,\eta,R)\leq\frac{\partial^2U}{\partial s^2}(s,\eta,R);
\end{equation}

(iii) $U(s,\eta,R)$ is strictly decreasing in $R$ for $\eta\leq s\leq1$, and
\begin{equation}
    \label{eq(2.9)}
    \lim_{R\to0^+}U(s,\eta,R)=\bar\sigma
\end{equation}
uniformly with respect to $s$ on $[\eta,1]$,
\begin{equation}
    \label{eq(2.10)}
    \lim_{R\to+\infty}U(s,\eta,R)=0
\end{equation}
uniformly with respect to $s$ on each closed interval contained in $[\eta,1)$.
\end{lemma}

\begin{proof}
Using the assumption (A1), it is easy to see that for any $R>0$ and any $0\leq\eta<1$, the constant functions $\underline{u}(s)\equiv0$ and $\bar{u}(s)\equiv\bar\sigma$ are sub- and supersolutions of the problem \eqref{eq(2.5)}, \eqref{eq(2.6)}, respectively. Thus, by the method of sub- and supersolutions, there exists at least one solution $u(s)=U(s,\eta,R)$ to this problem satisfying
\begin{equation}
    \label{eq(2.11)}
    0\leq U(s,\eta,R)\leq\bar{\sigma} \quad \text{for~all}~\eta\leq s\leq1,
\end{equation}
while the uniqueness follows from the maximum principle. Moreover, the uniqueness of solutions to the initial value problem implies that
\begin{equation}
    \label{eq(2.12)}
U(\eta,\eta,R)>0.
\end{equation}
Integrating \eqref{eq(2.5)} yields
\begin{equation}
    \label{eq(2.13)}
    \frac{\partial U}{\partial s}(s,\eta,R) = \frac{R^2}{s^2} \int_\eta^s l^2 f(U(l,\eta,R)) \dif l \quad \text{for}~\eta< s\leq 1,
\end{equation}
which together with \eqref{eq(2.11)} and the hypothesis (A1) gives
\begin{equation}
 \label{eq(2.14)}
\frac{\partial U}{\partial s}(s,\eta,R)\ge0\quad \text{for}~\eta< s\leq 1.
\end{equation}
\eqref{eq(2.7)} is then an immediate consequence of \eqref{eq(2.12)}-\eqref{eq(2.14)} and (A1). In particular, $\frac{\partial U}{\partial s}(1,\eta,R)>0$.
In view of the second boundary condition in \eqref{eq(2.6)}, we have $U(1,\eta,R)<\bar\sigma$; hence, the assertion (i) is proved.
The relation \eqref{eq(2.8)} can be easily derived from \eqref{eq(2.5)} and \eqref{eq(2.7)}.
Using \eqref{eq(2.13)} and L'Hospital' rule, we further compute
\begin{equation}
 \label{eq(2.15)}
\begin{split}
 \frac{\partial^2 U}{\partial s^2}(\eta,\eta,R)&=\lim_{s\to\eta^+}\frac{\frac{\partial U}{\partial s}(s,\eta,R)}{s-\eta}
 =\lim_{s\to\eta^+}\frac{R^2\int_\eta^s l^2 f(U(l,\eta,R)) \dif l}{s^2(s-\eta)}
 \\
 &=\begin{cases}
 \frac{R^2}3 f(U(\eta,\eta,R))\quad {\rm if}~\eta=0,\\
R^2 f(U(\eta,\eta,R))\quad {\rm if}~0<\eta<1.
 \end{cases}
 \end{split}
\end{equation}

We now show that
\begin{equation}
\label{eq(2.00)}
    \frac{\partial U}{\partial R}(s,\eta,R)<0 \quad \text{for~every}~\eta\le s\le1.
\end{equation}
First, we assume $\eta\in(0,1)$. Then $U(s,\eta,R)=\sigma(r,\rho,R)$ with $r=sR$, $\rho=\eta R$, and
\begin{equation}
\label{eq(2.16)}
\frac{\partial U}{\partial R}(s,\eta,R)=s\frac{\partial\sigma}{\partial r}(r,\rho,R)+\eta\frac{\partial\sigma}{\partial\rho}(r,\rho,R)+\frac{\partial\sigma}{\partial R}(r,\rho,R).
\end{equation}
In what follows, for notational simplicity we also denote $\frac{\partial z}{\partial r}(r,\rho,R)=z_r(r,\rho,R)=z'(r,\rho,R)$, $\frac{\partial^2z}{\partial r^2}(r,\rho,R)=z_{rr}(r,\rho,R)=z''(r,\rho,R)$ for a function of three variables $z(r,\rho,R)$. A simple calculation based on \eqref{eq(2.1)},
\eqref{eq(2.2)} yields
\begin{equation}
    \label{eq(2.17)}
    \begin{cases}
    &\sigma''_R(r,\rho,R)+\frac2r\sigma'_R(r,\rho,R)=f'(\sigma(r,\rho,R))\sigma_R(r,\rho,R) \quad \text{for}~\rho<r<R, \\
    &\sigma'_R(\rho,\rho,R)=0, \\
    &\sigma'_R(R,\rho,R)+\beta\sigma_R(R,\rho,R)=-\sigma''(R,\rho,R)-\beta\sigma'(R,\rho,R),
    \end{cases}
\end{equation}
\begin{equation}
    \label{eq(2.18)}
    \begin{cases}
    &\sigma''_\rho(r,\rho,R)+\frac2r\sigma'_\rho(r,\rho,R)=f'(\sigma(r,\rho,R)) \sigma_\rho(r,\rho,R) \quad \text{for}~\rho<r<R, \\
    &\sigma'_\rho(\rho,\rho,R)=-\sigma''(\rho,\rho,R), \\
    &\sigma'_\rho(R,\rho,R)+\beta\sigma_\rho(R,\rho,R)=0
    \end{cases}
\end{equation}
and
\begin{equation}
    \label{eq(2.19)}
    \begin{cases}
    &\sigma''_r(r,\rho,R)+\frac2r\sigma'_r(r,\rho,R)=f'(\sigma(r,\rho,R)) \sigma_r(r,\rho,R)+\frac2{r^2}\sigma'(r,\rho,R) \quad \text{for}~\rho<r<R, \\
    &\sigma'_r(\rho,\rho,R)=\sigma''(\rho,\rho,R), \\
    &\sigma'_r(R,\rho,R)+\beta\sigma_r(R,\rho,R)=\sigma''(R,\rho,R)+\beta\sigma'(R,\rho,R).
    \end{cases}
\end{equation}
Since it follows from \eqref{eq(2.7)}, \eqref{eq(2.8)} and \eqref{eq(2.15)} that
\begin{equation}
\label{eq(2.20)}
\sigma'(r,\rho,R)=\frac1{R}\frac{\partial U}{\partial s}(s,\eta,R)>0\quad\text{for}~\rho<r\le R,
\end{equation}
$$
\sigma''(r,\rho,R)=\frac1{R^2}\frac{\partial^2 U}{\partial s^2}(s,\eta,R)>0\quad\text{for}~\rho\le r\le R,
$$
applying the strong maximum principle we obtain
\begin{equation}
    \label{eq(2.21)}
\sigma_R(r,\rho,R)<0,\quad \sigma_\rho(r,\rho,R)>0\quad\text{for}~\rho\le r\le R.
\end{equation}
Let
$$
\Sigma(r,\rho,R)=\sigma_r(r,\rho,R)+\sigma_\rho(r,\rho,R)+\sigma_R(r,\rho,R).
$$
Then we deduce from \eqref{eq(2.17)}-\eqref{eq(2.19)} that
\begin{equation}
    \label{eq(2.22)}
    \begin{cases}
    &\Sigma''(r,\rho,R)+\frac2r\Sigma'(r,\rho,R)=f'(\sigma(r,\rho,R)) \Sigma(r,\rho,R)+\frac2{r^2}\sigma'(r,\rho,R) \quad \text{for}~\rho<r<R, \\
    &\Sigma'(\rho,\rho,R)=0, \\
    &\Sigma'(R,\rho,R)+\beta\Sigma(R,\rho,R)=0.
    \end{cases}
\end{equation}
Using the strong maximum principle again, we get
\begin{equation}
\label{eq(2.23)}
    \Sigma(r,\rho,R)<0 \quad \text{for}~ \rho \leq r \leq R.
\end{equation}
Combining \eqref{eq(2.16)}, \eqref{eq(2.20)}, \eqref{eq(2.21)} and \eqref{eq(2.23)}, we conclude \eqref{eq(2.00)}. Next, if $\eta=0$, then $U(s,0,R)=\sigma(r,0,R)$ and \eqref{eq(2.00)} can be verified in a similar manner.

It remains to prove \eqref{eq(2.9)} and \eqref{eq(2.10)}. By \eqref{eq(2.13)} and the second boundary condition in \eqref{eq(2.6)},
we have for any $s\in[\eta,1]$,
\begin{equation}
    \label{eq(2.24)}
\bar\sigma-U(s,\eta,R)
=\frac{1}{\beta R}\frac{\partial U}{\partial s}(1,\eta,R)+R^2\int_s^1\frac1{\tau^2}\int_\eta^\tau l^2 f(U(l,\eta,R)) \dif l \dif\tau,
\end{equation}
which combined with \eqref{eq(2.0)} and \eqref{eq(2.7)} gives
\begin{equation}
    \label{eq(2.25)}
\max_{s\in[\eta,1]}|U(s,\eta,R)-\bar\sigma|
\leq \frac{R}{3\beta}f(\bar\sigma)+R^2\int_{\eta}^1\frac1{\tau^2}\int_\eta^\tau l^2 f(\bar\sigma) \dif l \dif\tau.
\end{equation}
Hence our sending $R\to0^+$ in \eqref{eq(2.25)} yields \eqref{eq(2.9)}. On the other hand, it follows from \eqref{eq(2.24)} that
\begin{equation*}
0\leq\int_{\eta}^1\frac1{\tau^2}\int_\eta^\tau l^2 f(U(l,\eta,R)) \dif l \dif\tau\leq\frac{\bar\sigma}{R^2}.
\end{equation*}
Letting $R\to+\infty$, by Lebesgue's dominated convergence theorem we obtain
$$
\int_{\eta}^1\frac1{\tau^2}\int_\eta^\tau l^2 f\left(\lim_{R\to+\infty}U(l,\eta,R)\right) \dif l \dif\tau=0.
$$
In view of (A1), we arrive at $\lim_{R\to+\infty}U(s,\eta,R)=0$, a.e. $s\in(\eta,1)$. The monotonicity of $U(s,\eta,R)$ with respect to $s$ further implies \eqref{eq(2.10)}.
The proof is complete.
\end{proof}

\begin{remark}
We may not be able to expect $\lim_{R\to+\infty}U(1,\eta,R)=0$. As a matter of fact, in the typical case where $f(u)=u$, we compute
$$
U(s,\eta,R)=\frac{C}{s}[\eta R\cosh((s-\eta)R)+\sinh((s-\eta)R)] \quad \text{for}~\eta\leq s\leq 1
$$
with
$$
C=\frac{\beta\bar\sigma}{\left(\eta R-\frac1R+\beta\right)\sinh((1-\eta)R)+(1-\eta+\beta\eta R)\cosh((1-\eta)R)},
$$
which satisfies
$$
\lim_{R\to+\infty}U(1,\eta,R)=\frac{\beta\bar\sigma}{\beta+1}>0.
$$
\end{remark}

Define
\begin{equation}
\label{eq(2.26)}
    F(\eta,R)=U(\eta,\eta,R)-\sigma_D \quad \text{for}~ 0 \leq\eta<1,~R>0.
\end{equation}
Then by virtue of (A3) and the assertion (iii) of Lemma \ref{lem-2.1}, for each $0\leq\eta<1$, there exists a unique $R=R(\eta)>0$ such that $F(\eta,R)=0$. In particular, we denote $R(0)$ by $R_c$. Furthermore, one obtains from \eqref{eq(2.00)} and \eqref{eq(2.21)} that
\begin{equation}
\label{eq(2.51)}
R'(\eta)=\frac{U_\eta(\eta,\eta,R(\eta))}{-U_R(\eta,\eta,R(\eta))}=\frac{R(\eta)\sigma_\rho(\eta R(\eta),\eta R(\eta),R(\eta))}{-U_R(\eta,\eta,R(\eta))}>0\quad\text{for~any}~0<\eta<1,
\end{equation}
which implies that $R(\eta)$ is strictly increasing on $[0,1)$. While the next lemma states that the equation $F(\eta,R)=0$ also uniquely determines a function $\eta$ of $R$, whose domain is $[R_c,+\infty)$.

\begin{lemma}
\label{lem-2.2}
Suppose (A1), (A3) are satisfied. Then the following assertions hold:

(i) For each $R\geq R_c$, there exists a unique $\eta=\eta(R)\in [0,1)$ such that $F(\eta,R)=0$.

(ii) The mapping $R\mapsto\eta(R)$ is strictly increasing, and
\begin{equation}
    \label{eq(2.27)}
    \lim_{R\to+\infty}\eta(R)=1.
\end{equation}

(iii) For any $0<R<R_c$ and $0 \leq \eta<1$, $F(\eta,R)>0$.
\end{lemma}

\begin{proof}
(i) We begin by showing the existence. Evidently, $F(0,R_c)=0$ and $F(0,R)<0$ for $R>R_c$. Since \eqref{eq(2.13)} and \eqref{eq(2.24)} imply
\begin{equation}
\label{eq(2.01)}
    U(\eta,\eta,R)=\bar{\sigma}-\frac{R}{\beta} \int_\eta^1 l^2f(U(l,\eta,R)) \dif l
-R^2\int_\eta^1\frac1{\tau^2}\int_\eta^\tau l^2f(U(l,\eta,R))\dif l\dif \tau \quad\text{for~all}~\eta\in[0,1),
\end{equation}
we see from \eqref{eq(2.0)} and (A3) that
\begin{equation*}
    \lim_{\eta\to1^-}F(\eta,R)=\bar{\sigma}-\sigma_D>0.
\end{equation*}
Thus, by the continuity of $F(\eta,R)$ we derive that for fixed $R>R_c$, there exists at least one $\eta\in(0,1)$ such that $F(\eta,R)=0$. Next, the uniqueness follows from $R(\eta)$ being strictly monotone.

(ii) In view of \eqref{eq(2.51)},
$\eta'(R)>0$ for all $R>R_c$; thus, $\eta(R)$ is strictly increasing on $[R_c,+\infty)$ and $\lim_{R\to+\infty}\eta(R)$ exists. By \eqref{eq(2.01)}, we have
\begin{equation}
\label{eq(2.52)}
\begin{aligned}
\bar\sigma-\sigma_D&=\frac{R}{\beta} \int_{\eta(R)}^1 l^2f(U(l,\eta(R),R)) \dif l
+R^2\int_{\eta(R)}^1\frac1{\tau^2}\int_{\eta(R)}^\tau l^2f(U(l,\eta(R),R))\dif l\dif \tau
\\
&\geq\frac{R}{\beta}f(\sigma_D) \int_{\eta(R)}^1 l^2\dif l.
\end{aligned}
\end{equation}
Hence, sending $R\to+\infty$ in \eqref{eq(2.52)} yields \eqref{eq(2.27)}.

(iii) Noticing that $R(\eta)\geq R_c$ for each $\eta\in[0,1)$ and
$U(\eta,\eta,R)$ is strictly decreasing in $R$, we immediately arrive at
$$
U(\eta,\eta,R)>U(\eta,\eta,R_c)\geq U(\eta,\eta,R(\eta))=\sigma_D\quad\text{for~all}~\eta\in[0,1)~\text{and}~ R\in(0,R_c),
$$
which proves the assertion (iii) and completes the proof of the lemma.
\end{proof}

Given $R>0$, by the uniqueness of solutions of the problem \eqref{eq-1}-\eqref{eq-3}, we see from Lemmas \ref{lem-2.1} and \ref{lem-2.2} that the solution of \eqref{eq-1}-\eqref{eq-3}, denoted by $\sigma(r,R)$, is as follows: if $0<R\leq R_c$, then
\begin{equation}
\label{eq(2.43)}
\sigma(r,R)=U\left(\frac{r}{R},0,R\right)\quad
\text{for}~0\leq r\leq R,
\end{equation}
and if $R>R_c$, then
\begin{equation}
    \label{eq(2.30)}
 \sigma(r,R) =
  \begin{cases}
  U\left( \frac{r}{R},\eta(R),R\right)\quad& \text{if}~\rho(R)<r\leq R, \\
  \sigma_D\quad &\text{if}~ 0\leq r\leq\rho(R),
  \end{cases}
\end{equation}
where $\rho(R)=\eta(R)R$.

\begin{remark}
\label{rem-2.2}
\eqref{eq(2.43)} and \eqref{eq(2.30)} imply that if the tumor radius $R\leq R_c$ ($R_c$ depends on $\bar\sigma$, $\sigma_D$ and $\beta$), then the tumor is nonnecrotic, and if $R>R_c$ then the tumor is necrotic.
\end{remark}

\begin{remark}
\label{rem-2.1}
It is shown in \citep[Lemma 2.1]{Z-C(18)} that the function $\sigma(r,R)$ given by \eqref{eq(2.43)} has the following properties:
\begin{equation}
\label{eq(2.47)}
0\leq\sigma_r(r,R)\leq f(\bar\sigma)\frac{r}3,\quad
-f(\bar\sigma)\left(\frac1\beta+\frac{R}3\right)\leq \sigma_R(r,R)\leq0
\end{equation}
for all $R>0$ and $0\leq r\leq R$, and
\begin{equation}
\label{eq(2.48)}
\frac1r\sigma_r(r,R)\leq \sigma_{rr}(r,R)\leq f(\bar\sigma)
\end{equation}
for all $R>0$ and $0<r\leq R$. In fact, one can further obtain $\sigma_r(r,R)>0$ for all $R>0$ and $r>0$, and if we denote $M=\sup_{\sigma\geq0}f'(\sigma)$, then a simple calculation based on the assumption (A1) yields
\begin{equation}
\label{eq(2.49)}
\sigma(r,R)\leq\sigma(R,R)e^{\frac{M}6(r^2-R^2)}
\end{equation}
for all $R>0$ and $r>R$. Similarly, for the function $\sigma(r,R)$ given by \eqref{eq(2.30)} we also have \eqref{eq(2.47)} and \eqref{eq(2.48)} for all $R>R_c$ and $\rho(R)\leq r\leq R$,  \eqref{eq(2.49)} for all $R>R_c$ and $r>R$.
\end{remark}

Substituting \eqref{eq(2.43)} or \eqref{eq(2.30)} into \eqref{eq-4}, we obtain the equation for the dormant tumor radius $R_s$:
\begin{equation}
\label{eq(2.39)}
G(R)=0,
\end{equation}
where
\begin{equation}
    \label{eq(2.31)}
    G(R)=
        \begin{cases}
        \int_0^1g(U(s,0,R))s^2\dif s\quad\text{if}~0<R\leq R_c,\\
        \int_{\eta(R)}^1g(U(s,\eta(R),R))s^2\dif s-\frac{\nu}{3}\eta^3(R) \quad\text{if}~R>R_c.
        \end{cases}
\end{equation}

\begin{lemma}
\label{lem-2.3}
Assume that (A1)-(A3) hold. Then the function $G(R)$ defined by \eqref{eq(2.31)} possesses the following properties:

(i) $G(R)$ is continuous and strictly decreasing on $(0,+\infty)$.

(ii)
\begin{equation*}
%    \label{eq(2.32)}
    \lim_{R\to0^+}G(R)=\frac{g(\bar\sigma)}{3},\quad
    \lim_{R\to+\infty}G(R)=-\frac{\nu}{3}.
\end{equation*}

(iii) If $R_c$ is regarded as a function of $\bar\sigma$, then there exists a number $\sigma^*>\tilde\sigma$ such that
\begin{equation}
    \label{eq(2.41)}
    G(R_c(\bar\sigma))
          \begin{cases}
          >0\quad\text{if}~\bar\sigma>\sigma^*,\\
          =0\quad\text{if}~\bar\sigma=\sigma^*,\\
          <0\quad\text{if}~\sigma_D<\bar\sigma<\sigma^*.
          \end{cases}
\end{equation}
\end{lemma}

\begin{proof}
It is easy to see that $G(R)$ is continuous on $(0,+\infty)$. By \eqref{eq(2.00)} and the assumption (A2), we find for $0<R<R_c$,
\begin{equation*}
%\label{eq(2.46)}
    G'(R)=\int_0^1 g'(U(s,0,R))\frac{\partial U}{\partial R}(s,0,R)s^2\dif s<0.
\end{equation*}
When $R>R_c$, we write $U(s,\eta(R),R)$ as $V(s,R)$, or equivalently, $V(s,R)=\sigma(r,R)$ for $R>R_c$ and $\eta(R)\leq s\leq 1$, where $\sigma(r,R)$ is given by \eqref{eq(2.30)}. Then an argument similar to that used in obtaining \eqref{eq(2.00)} yields
\begin{equation}
\label{eq(2.50)}
    \frac{\partial V}{\partial R}(\eta(R),R)=0,\quad \frac{\partial V}{\partial R}(s,R)<0 \quad \text{for}~\eta(R)<s\leq1.
\end{equation}
Consequently, a combination of \eqref{eq(2.50)}, the assumptions (A2) and (A3) leads to
\begin{equation*}
%\label{eq(2.42)}
    G'(R)=\int_{\eta(R)}^1 g'(V(s,R)) \frac{\partial V}{\partial R}(s,R) s^2 \dif s-[g(\sigma_D)+\nu]\eta^2(R) \eta'(R)<0 \quad \text{for}~R>R_c.
\end{equation*}
The assertion (i) is thus proved.
The assertion (ii) immediately follows from \eqref{eq(2.9)} and \eqref{eq(2.27)}.

We now proceed with the proof of the assertion (iii). Let $R_c=R_c(\bar\sigma)$ for $\bar\sigma>\sigma_D$ and write
\begin{equation*}
    W(s,\bar\sigma)=U(s,0,R_c(\bar{\sigma})), \quad
    \mathcal{G}(\bar\sigma)=G(R_c(\bar\sigma))=\int_0^1g(W(s,\bar\sigma))s^2 \dif s.
\end{equation*}
Then there holds
\begin{equation}
    \label{eq(2.33)}
 R'_c(\bar\sigma)>0 \quad \text{for}~\bar{\sigma}>\sigma_D.
\end{equation}
In fact, setting $\sigma(r,\bar\sigma)=W(s,\bar\sigma)$ with $r=sR_c(\bar\sigma)$, we deduce from \eqref{eq(2.03)}-\eqref{eq(2.05)} that
\begin{equation}
\label{eq(2.34)}
\begin{cases}
\frac{\partial^2\sigma}{\partial r^2}(r,\bar\sigma)+\frac2r\frac{\partial\sigma}{\partial r}(r,\bar\sigma)=f(\sigma(r,\bar\sigma))\quad\text{for}~0<r<R_c(\bar\sigma),\\
\frac{\partial\sigma}{\partial r}(0,\bar\sigma)=0,\quad \sigma(0,\bar\sigma)=\sigma_D,\\
\frac{\partial\sigma}{\partial r}(R_c(\bar\sigma),\bar\sigma)+\beta[\sigma(R_c(\bar\sigma),\bar\sigma)-\bar\sigma]=0.
\end{cases}
\end{equation}
Differentiating \eqref{eq(2.34)} with respect to $\bar\sigma$ and by the uniqueness of solutions to the initial value problem, we arrive at
\begin{equation*}
\frac{\partial\sigma}{\partial\bar\sigma}(r,\bar\sigma)\equiv0\quad\text{for~every}~0\leq r\leq R_c(\bar\sigma)
\end{equation*}
and
\begin{equation*}
 R'_c(\bar\sigma)=\frac{\beta}{\frac{\partial^2\sigma}{\partial r^2}(R_c(\bar\sigma),\bar\sigma)
+\beta\frac{\partial\sigma}{\partial r}(R_c(\bar\sigma),\bar\sigma)}>0,
\end{equation*}
which proves \eqref{eq(2.33)}.
Furthermore,
\begin{equation}
\label{eq(2.07)}
\frac{\partial W}{\partial\bar\sigma}(s,\bar\sigma)
=s R'_c(\bar\sigma)\frac{\partial\sigma}{\partial r}(r,\bar\sigma)+\frac{\partial\sigma}{\partial\bar\sigma}(r,\bar\sigma)
=s R'_c(\bar\sigma)\frac{\partial\sigma}{\partial r}(r,\bar\sigma)>0
\end{equation}
for $0<s\leq1$, which implies
\begin{equation}
\label{eq(2.35)}
    \mathcal{G}'(\bar\sigma)=\int_0^1 g'(W(s,\bar\sigma))\frac{\partial W}{\partial\bar\sigma}(s,\bar\sigma) s^2 \dif s>0\quad\text{for}~\bar\sigma>\sigma_D.
\end{equation}
From \eqref{eq(2.0)} and (A2) we see $\mathcal{G}(\tilde\sigma)<0$. On the other hand, since
$$
 \frac{\partial W}{\partial s}(s,\bar\sigma)=\frac1{s^2}\int_0^sR_c^2(\bar\sigma)l^2f(W(l,\bar\sigma))\dif l\quad\text{for}~0<s\leq1,
$$
recalling that $M=\sup_{\sigma\geq0}f'(\sigma)$, using (A1) we derive
\begin{equation}
\label{eq(2.37)}
\frac13R_c^2(\bar\sigma)f(\sigma_D)s\leq
\frac{\partial W}{\partial s}(s,\bar\sigma)
\leq\frac{M}3R_c^2(\bar\sigma)sW(s,\bar\sigma)\quad\text{for}~0\leq s\leq1.
\end{equation}
As a result,
\begin{equation}
\label{eq(2.38)}
\sigma_D+\frac16R_c^2(\bar\sigma)f(\sigma_D)s^2\leq
    W(s,\bar\sigma)\leq\sigma_De^{\frac{M}6R_c^2(\bar\sigma)s^2}\quad\text{for~all}~0\leq s\leq1.
\end{equation}
Making use of \eqref{eq(2.37)}, \eqref{eq(2.38)} and the relation
\begin{equation*}
\beta\bar\sigma=\frac1{R_c(\bar\sigma)}\frac{\partial W}{\partial s}(1,\bar\sigma)+\beta W(1,\bar\sigma),
\end{equation*}
we obtain
\begin{equation*}
%\label{eq(2.36)}
\lim_{\bar\sigma\to\sigma_D^+}R_c(\bar\sigma)=0,\quad\lim_{\bar\sigma\to+\infty} R_c(\bar\sigma)=+\infty,
\end{equation*}
which in turn leads to
\begin{equation}
\label{eq(2.40)}
\lim_{\bar\sigma\to+\infty}W(s,\bar\sigma)=+\infty\quad\text{for}~0<s\leq1.
\end{equation}
Based on (A2) and \eqref{eq(2.07)}, Levi's theorem gives
\begin{equation*}
\lim_{\bar\sigma\to+\infty}\int^1_0[g(W(s,\bar\sigma))-g(\sigma_D)]s^2\dif s=
\int^1_0\lim_{\bar\sigma\to+\infty}[g(W(s,\bar\sigma))-g(\sigma_D)]s^2\dif s.
\end{equation*}
Hence, it follows from \eqref{eq(2.40)} and (A2) that 
\begin{equation*}
\lim_{\bar\sigma\to+\infty}\mathcal{G}(\bar\sigma)=
\int^1_0\lim_{\bar\sigma\to+\infty}g(W(s,\bar\sigma))s^2\dif s>0.
\end{equation*}
We therefore conclude the assertion (iii) by \eqref{eq(2.35)} and complete the proof of the lemma.
\end{proof}

Now we are ready to state the main result of this section.

\begin{theorem}
\label{thm-1}
Let (A1)-(A3) hold. Then the system \eqref{eq(1.1)}-\eqref{eq(1.5)} admits a
unique stationary solution $(\sigma_s(r), R_s)$ if and only if $\bar\sigma>\tilde\sigma$. If $\bar\sigma>\tilde\sigma$, then $R_s$ is the unique positive root of the equation \eqref{eq(2.39)}, and $\sigma_s(r)=\sigma(r,R_s)$ defined by \eqref{eq(2.43)} or \eqref{eq(2.30)}. Furthermore, there exists a positive number $\sigma^*>\tilde\sigma$ such that if $\bar\sigma>\sigma^*$, then $R_s>R_c$, $\sigma_s(r)$ is given by \eqref{eq(2.30)}, and accordingly, the dormant tumor has a necrotic core with radius $\rho(R_s)$, while if $\tilde\sigma<\bar\sigma\leq\sigma^*$, then $R_s\leq R_c$, $\sigma_s(r)$ is given by \eqref{eq(2.43)}, and the dormant tumor does not have a necrotic core.
\end{theorem}

\section{Transient solutions}

In this section, we study transient solutions of the problem \eqref{eq(1.1)}-\eqref{eq(1.5)}. The first result concerns the global existence and asymptotic behavior of transient solutions.

\begin{theorem}
\label{thm-2}
Let the assumptions (A1)-(A3) be satisfied. Then for any $R_0>0$, the problem \eqref{eq(1.1)}-\eqref{eq(1.5)} has a unique solution $(\sigma(r,t),R(t))$ ($R(t)>0$, $0\leq r\leq R(t)$) for all $t>0$. Moreover, if $\bar\sigma\leq\tilde\sigma$, then $\lim_{t\to+\infty}R(t)=0$,
while if $\bar\sigma>\tilde\sigma$, then $\lim_{t\to+\infty}R(t)=R_s$ and $\lim_{t\to+\infty}\max_{0\leq r\leq R(t)}|\sigma(r,t)-\sigma_s(r)|=0$.
\end{theorem}

\begin{proof}
Given $R(t)>0$, we know from Section 2 that \eqref{eq(1.1)}-\eqref{eq(1.3)} admits a unique solution $\sigma(r,R(t))$ (see \eqref{eq(2.43)}, \eqref{eq(2.30)}). Then by \eqref{eq(1.4)} and \eqref{eq(1.5)}, $R(t)$ can be determined by solving the following initial value problem for a first order ordinary differential equation
\begin{equation}
\label{eq(3.1)}
\begin{cases}
R'(t)=R(t)G(R(t))\quad\text{for}~t>0,\\
R(0)=R_0,
\end{cases}
\end{equation}
where the function $G$ is given by \eqref{eq(2.31)}.
The assertions (i), (ii) of Lemma \ref{lem-2.3} imply that for any $R_0>0$, the problem \eqref{eq(3.1)} has a unique solution $R(t)$ satisfying
\begin{equation}
\label{eq(3.2)}
R_0 e^{-\frac{\nu}3t} \leq R(t) \leq R_0 e^{\frac{g(\bar\sigma)}3t} \quad \text{for~every}~t\geq0.
\end{equation}
Accordingly, \eqref{eq(1.1)}-\eqref{eq(1.5)} has a unique solution $(\sigma(r,t),R(t))=(\sigma(r,R(t)),R(t))$ for all $t\geq0$.

If $\bar\sigma<\tilde\sigma$, then $g(\bar\sigma)<0$ and $\lim_{t\to+\infty}R(t)=0$ immediately follows from \eqref{eq(3.2)}.
If $\bar\sigma=\tilde\sigma$, then $G<0$ on $(0,+\infty)$ by the assertions (i), (ii) of Lemma \ref{lem-2.3}, and thus $R'(t)<0$ for $t>0$. We argue by contradiction and suppose that
$\lim_{t\to+\infty}R(t)=L$ for some $L\in(0,R_0)$. Then $R(t)\leq R_0 e^{G(L)t}$ for $t\geq0$, reaching a contradiction. Finally, if $\bar\sigma>\tilde\sigma$,
then $G(R_s)=0$. Using this and the fact that $G(R)$ is strictly decreasing on $(0,+\infty)$, we can similarly obtain $\lim_{t\to+\infty}R(t)=R_s$, which together with Remark \ref{rem-2.1} gives $\lim_{t\to+\infty}\max_{0\leq r\leq R(t)}|\sigma(r,t)-\sigma_s(r)|=0$.
The proof is complete.
\end{proof}

The next result suggests that mutual transition may exist between the nonnecrotic and necrotic phases in the growth of tumors.

\begin{theorem}
\label{thm-3}
Assume that (A1)-(A3) are fulfilled. Then the following conclusions hold.

(i) Let $\tilde\sigma<\bar\sigma<\sigma^*$. If $R_0\leq R_c$, then the tumor is in the nonnecrotic phase for all $t\geq0$, while if $R_0>R_c$, then there exists a finite number $T>0$ such that the tumor is in the necrotic phase for $0\leq t<T$, and in the nonnecrotic phase for $t\geq T$.

(ii) Let $\bar\sigma>\sigma^*$. If $R_0<R_c$, then there exists a finite number $T>0$ such that the tumor is in the nonnecrotic phase for $0\leq t\leq T$, and in the necrotic phase for $t>T$, while if $R_0\geq R_c$, then the tumor is in the necrotic phase for all $t>0$.
\end{theorem}

The proof of Theorem \ref{thm-3} relies essentially on the following four facts: an evolutionary tumor is in the nonnecrotic phase if its radius $R(t)\leq R_c$ and in the necrotic phase if $R(t)>R_c$, ensured by Remark \ref{rem-2.2};  an evolutionary tumor converges to the dormant tumor when $\bar\sigma>\tilde\sigma$ ensured by Theorem \ref{thm-2}; the dormant tumor is nonnecrotic and its radius $R_s<R_c$ if $\tilde\sigma<\bar\sigma<\sigma^*$, whereas the dormant tumor is necrotic and $R_s>R_c$ if $\bar\sigma>\sigma^*$, ensured by Theorem \ref{thm-1}; the function $G$ in the evolution equation \eqref{eq(3.1)} is strictly decreasing on $(0,+\infty)$ ensured by Lemma \ref{lem-2.3}. The details are omitted for brevity.

\begin{remark}
\label{rem-3.1}
Theorem \ref{thm-3} can be interpreted in the biological context as follows. If the dormant tumor is nonnecrotic, then an initially small tumor will always be nonnecrotic and an initially large tumor will be nonnecrotic from a finite time; if the dormant tumor has a necrotic core, then an initially small tumor will form a necrotic core at a finite time, and an initially large tumor will always have a necrotic core.
\end{remark}

\section*{Acknowledgments}
This work was partly supported by the National Natural Science Foundation of China (No. 11601200, No. 11861038
and No. 11771156), and Graduate Innovation Fund of Jiangxi Normal University.

\end{document}